\documentclass[12pt, reqno]{amsart}
\usepackage{amsmath,amsfonts,amssymb,amsthm,enumerate,appendix,caption}
\usepackage{cmap,mathtools}
\usepackage{paralist,bm}
\usepackage{graphics} 
\usepackage{epsfig} 
\usepackage{graphicx}
\usepackage{braket}
\usepackage{epstopdf}
 \usepackage[colorlinks=true]{hyperref}
\usepackage{xcolor}
\definecolor{myblue}{rgb}{0.0, 0.0, 1.0}
\definecolor{mygreen}{rgb}{0.01,0.75,0.20}
\hypersetup{ linkcolor=myblue, colorlinks=true, urlcolor=black, citecolor=red}
\usepackage{hyperref}

\newtheorem{theorem}{Theorem}[section]

\newtheorem{cor}[theorem]{Corollary}
\newtheorem{thm}[theorem]{Theorem}
\newtheorem{lem}[theorem]{Lemma}
\newtheorem{lemma}[theorem]{Lemma}
\newtheorem{proposition}[theorem]{Proposition}

\theoremstyle{definition}
\newtheorem{definition}[theorem]{Definition}

\newtheorem{remark}[theorem]{Remark}

\newtheorem{defin}[theorem]{Definition}
\newtheorem{example}[theorem]{Example}

\theoremstyle{definition}

\usepackage[margin=1in]{geometry}

\numberwithin{equation}{section}

                         \def\vge{\varepsilon}

     \def\Gd{\Delta}

\newcommand{\eat}[1]{}

\DeclarePairedDelimiter\norm{\lVert}{\rVert}%
\makeatletter
\let\oldnorm\norm
\def\norm{\@ifstar{\oldnorm}{\oldnorm*}}

\newcommand{\Om} {\Omega}

\newcommand\restr[2]{{
  \left.\kern-\nulldelimiterspace 
  #1 
  \right|_{#2} 
  }}

\def\w{{\widetilde w}}

\def\w2{{W^{1,2}_0(\Om)}}
\def\hh2{{H^1_0(\Om)}}

\def\C{{\mathcal C}}

\def\E{{\mathcal E}}

\def\N{{\mathbb N}}
\def\F{{\mathcal F}}

\def\R{{\mathbb R}}

\def\ws2{{\F_{\frac{N}{2}}}}

\def\c1{{\C_c^1}}

\def\H{{\mathcal{H}}}

\newcommand\sol[1]{{{#1}}}
\newcommand\Q{Q}

\renewcommand{\L}{\Delta_{p}}
\newcommand{\ph}{\varphi}
\renewcommand{\H}{\mathcal{H}}

\renewcommand{\E}{E_{{\sol{u}}}}

\newcommand{\Hmm}[1]{\leavevmode{\marginpar{\tiny%
			$\hbox to 0mm{\hspace*{-0.5mm}$\leftarrow$\hss}%
			\vcenter{\vrule depth 0.1mm height 0.1mm width \the\marginparwidth}%
			\hbox to
			0mm{\hss$\rightarrow$\hspace*{-0.5mm}}$\\\relax\raggedright #1}}}

\begin{document}
	\title[Landis Conjecture]{On the Landis Conjecture for Positive Quasi-linear Operators on Graphs}
	
	\author {Ujjal Das}
	
	\address {Ujjal Das, BCAM-Basque Center for Applied Mathematics, Bilbao, Spain}
	
	\email {getujjaldas@gmail.com, udas@bcamath.org}
\author[M.~Keller]{Matthias Keller}
\address{Matthias~Keller,  Institut f\"ur Mathematik, Universit\"at Potsdam,
	14476  Potsdam, Germany}
\email{matthias.keller@uni-potsdam.de}	

	\author{Yehuda Pinchover}
	\address{Yehuda Pinchover,
		Department of Mathematics, Technion - Israel Institute of
		Technology,   Haifa, Israel}
	\email{pincho@technion.ac.il}
	
	\begin{abstract}
We prove a Landis type unique continuation result for positive quasi-linear operators on graphs. Specifically, we give decay criteria that ensures when a harmonic function for a positive {quasilinear Schr\"odinger}  operator with potential less than 1 is trivially zero. The assumption of positivity of the operator allows the application of criticality theory such as the Liouville comparison theorem. Furthermore, our results fundamentally build on the so called simplified energy. As an application we discuss the case of model graphs and in particular regular trees.



\medskip

\begin{center}
	{{\em Dedicated to the memory of Ha\"im Brezis}}
\end{center}

%
%
%
%
	\end{abstract}
\maketitle

\section{Introduction}
The Landis conjecture, originating in the study of elliptic partial differential equations, concerns the decay properties of solutions to second-order elliptic equations and their relation to uniqueness. Originally posed for the continuum Schr\"odinger equation, it states that if a bounded solution of $\Delta u + V u = 0$ in Euclidean space  with $|V|\leq 1$ decays faster than $\exp(-c|x|)$ for some {$c > 1$,} then $u = 0$ \cite{Landis2}. While true in dimension one \cite{Rossi}, Meshkov disproved it for complex-valued potentials in higher dimensions, constructing a nontrivial {complex-valued potential $V$ and a} solution $u$ decaying like $\exp(-c|x|^{4/3})$ \cite{Meshkov}. For real-valued potentials, the conjecture remains open, though partial results exist under {either stronger decay assumptions on the solution $u$} {or positivity assumptions of the Schr\"odinger operator, (see for example, \cite{DP24} and references therein).} The real-valued case started gaining attention since the celebrated work of Bourgain--Kenig \cite{BK}, where they derived a quantitative decay estimate of normalized solution of the above Schr\"odinger equation. This decay estimate in turn proved a weaker version of Landis’ claim. Since then, there has been growing attention on improving such a quantitative decay estimate and obtaining a sharper result than Bourgain--Kenig's, see \cite{Davey1,Davey, KSW, LMNN} and references therein. Deriving such estimates is challenging; it involves tools such as, Carleman-type inequalities, Hadamard
three-ball inequalities, and quasi-conformal techniques. Most of the improved Bourgain-Kenig estimates are
available only in dimension two. The sharpest result is due to the breakthrough by Logunov-Malinikova-Nadirashvili-Nazarov \cite{LMNN}.  Landis-type results have been studied for various operators, such as Schr\"odinger operators with drift \cite{Davey,KSW},  Dirac
 operators \cite{Cassano},  fractional Schr\"odinger operators \cite{KOW,RW19},  the time-dependent
 Schr\"odinger operators \cite{EKPV10,EKPV16}. 

  Beyond the Euclidean space $\R^d$, the conjecture has been addressed on manifolds \cite{PPV} and also on the Euclidean lattice \cite{FV,LM}. We also refer to the recent article \cite{FBRS24}, where the Landis-type results are obtained on meshes $(h\mathbb{Z})^d$,
 and the authors analyzed the behavior as the mesh-size $h$ decreases to zero. We refer to a recent review on Landis' conjecture for a comprehensive list of
 references \cite{FBSR24}.

 Recently, we have studied this conjecture  on general infinite discrete graphs under the additional assumption that the underlying Schr\"odinger
 operator is {\it positive} in the sense of the quadratic form \cite{DKP}. Simultaneously, quasilinear versions have been explored in the continuum for operators like the $p$-Laplacian, leveraging criticality theory and Liouville comparison principles, see \cite{DP24}. Although the positivity of the operator is a crucial restriction in these results, under this additional assumption, we could prove
 a sharp decay criterion for the validity of the Landis conjecture. We emphasize that there are Landis-type unique continuation results where the authors implicitly assumed the positivity of the underlying operator \cite{ABG,KSW,Sirakov}.


This paper bridges these developments by establishing the Landis conjecture for quasilinear operators on infinite graphs. We consider quasilinear Schr\"odinger operators  $\mathcal{H} = \Delta_p + V$ for the  $p$-Laplacian  $ \Delta_p $. Our main result shows that when $\mathcal{H} \geq 0$ with $V \leq 1$, any $\mathcal{H}$-harmonic function {satisfying certain (sharp) weak decay conditions must vanish.} The proof relies fundamentally on a Liouville comparison principle for {positive} quasilinear operators, which transfers criticality and ground state properties between operators under weighted norm constraints.

We demonstrate the applicability of our framework through two key examples: subcritical model graphs where explicit Green functions asymptotics are available, and $d$-regular trees where we obtain sharp decay rates $O(d^{-|x|/p})$. These illustrate how our unified approach extends the discrete linear Landis conjecture to the quasilinear setting while maintaining sharp conditions.

The paper is structured as follows. Section \ref{sec-prelim} reviews preliminaries on graph $p$-Laplacians, energy functionals, and criticality theory, culminating in the Liouville comparison principle. Section 3 presents and proves our main Landis-type theorem. Section 4 applies these results to model graphs and regular trees, with examples highlighting the sharpness of decay conditions.

\section{Preliminaries and Liouville Comparison Principle}\label{sec-prelim}

We begin this section with a few remarks about notation. With a slight abuse of notation, we identify constants with the corresponding constant function at times which is mainly applied to the constants $ 0 $ and $ 1 $. For example, we write $ f\ge 0 $ for  a function which is pointwise larger than or equal to $ 0 $. In the case when $ f $ is additionally not trivial, then we say $ f $ is \emph{positive}. Moreover, if $ f>0 $, then we say $ f $ is \emph{strictly positive}.  We denote the characteristic function of a set $ A $ by $ 1_{A} $ and also write $ 1=1_{X} $. For real valued functions $ f $ and $g$, we write $ f\wedge g=\min\{f,g\} $ and $ f\vee g=\max\{f,g\} $,  and we denote the positive and negative part of $f$ by $ f_{\pm}=({\pm}f)\vee0 $. Furthermore, we denote $f\lesssim g$ if there is a constant $C\ge 0$ such that $f\leq C g$ and $f\asymp g$ if $f\lesssim g$  and $g\lesssim f$.
Finally, in this text, $ C $ denotes a constant which depends only on $ p $ and may change from line to line.

\subsection{Graphs, $p$-Laplacian, energy functionals} Let $X$ be  an infinite countable set    equip\-ped with the discrete topology. We denote $C(X)={\{f:X \rightarrow \R \}}$ and
$C_c(X)=\{f\in C(X) \mid {\rm{supp}}(f) \Subset X\}$, where  $ K\Subset X$ means that $ K $ is a compact (i.e., finite) subset of $ X $.
Given a positive or an absolutely summable function $ f $ on a discrete set $ A $, we write 
\begin{align*}
	\sum_{A}f:=\sum_{a\in A} f(a)
\end{align*}
as it is common terminology for integrals. Below the set $ A $ will typically be a subset of either $ X $ or $ X\times X $.  A  function $m:X\to [0,\infty)  $ induces a measure on $ X $ by letting $$  m(A)=\sum_{A}m,\qquad A\subseteq X  .$$ 
We assume that $ m $ is a \emph{measure of full support}, i.e., the {density} function $ m $ is strictly positive. 
For $ q\in[1,\infty) $, we denote the Banach spaces
\begin{align*}
	\ell^{q}(X,m)=\big\{ f:X\to \R\mid  \|f\|_{q}:=\big(\sum_{X}m|f|^q\big)^{1/q} <\infty \big\},
\end{align*}
and  $ \ell^{\infty}(X)= \{f:X\to \R\mid  \|f\|_{\infty}:=\sup_{X}|f|<\infty\}$. 


Let $b$ be a connected graph over $(X,m)$, i.e., $ b:X\times X\to [0,\infty) $ is symmetric, has zero diagonal,  satisfies
\begin{align*}
	\sum_{y\in X}b(x,y)<\infty\qquad x\in X,
\end{align*} 
and for every $ x,y\!\in \! X $ there are $ x\!= \!x_{0} \! \sim\!\ldots \!\sim\! x_{n}\!=\!y $, where we write $ u\!\sim\! v $ whenever $ b(u,v)\!>\!0 $ in which case we call $ u $ and $ v $ adjacent. A graph $ b $ is called \emph{locally finite} if $ \#\{y\in X\mid y\sim x\} <\infty$ for all $ x \in X$. 
While we do not need this assumption for the main results of this paper, it is worthwhile to look at the special case of locally finite graphs from time to time.

For  $ p\in (1,\infty) $ and a potential $ V: X\to \R $,  we define the energy functional $ \Q=\Q_{p,b,m,V} $ on $ C_{c}(X) $ {by}
$$\Q(\varphi)=\frac{1}{2}\displaystyle \sum_{x,y \in X}  b(x,y)|\nabla_{xy}\varphi|^p + \sum_{x \in X} m(x)V(x)|\varphi(x)|^p,  \qquad \varphi \in  C_c(X), $$
where,  for $ x,y\in X $ and $ f\in C(X) $, the difference operator is
$$ \nabla_{xy}f:=f(x)-f(y). $$  
The energy functional $ \Q $ gives rise to the 
{\em weighted $p$-Laplacian} $ \L =\Delta_{p,b,m}  $  which acts on the following space of functions
\begin{align*}
	\mathcal{F}(X):=\mathcal{F}_{b}(X):= \big\{ f\in C(X)\mid \sum_{y \in X} b(x,y) |\nabla_{xy} {f}|^{p-1}<\infty\mbox{ for all }x\in X  \big\}
\end{align*} 
as
$$ \L[ f](x):=\frac{1}{m(x)}\sum_{y \in X} b(x,y) (\nabla_{xy} f)^{\langle p-1\rangle}, $$
where  for $a\in \mathbb{R}$ and $r>0$
\begin{equation*}
	a^{\langle r\rangle}=a|a|^{r-1}
\end{equation*}
whenever $a\neq 0$ and $0^{\langle r\rangle}=0$ otherwise.
 
Clearly, in the case of locally finite graphs one has $ \mathcal{F}(X)=C(X) $.

For a potential $ W:X\to \mathbb{R} $, we denote for $ u\in C(X) $ and $ x\in X $
\begin{align*}
	W[u](x):=W(x)(u(x))^{\langle p-1\rangle}.
\end{align*}
The {\em quasilinear {Schr\"odinger} operator}  associated to $Q$ will be denoted by $\mathcal{H}=\Delta_p+V$ which acts on $\mathcal{F}_{b}(X)$ as
\begin{equation*}
	\mathcal{H}[f](x) = \Delta_{p}f(x) + V(x)|f(x)|^{p-2}f(x),\qquad x\in X.
\end{equation*}
\begin{remark}
	
The continuum counterpart of $\L$  is referred to as the 
pseudo $p$-Laplacian in $\mathbb{R}^d$,  which acts as $\varphi \mapsto - \sum_{n=1}^n\partial_i(|\partial_i\varphi|^{p-2}\partial_i\varphi) $, and  is sometimes denoted in the literature by $-\widetilde{\L}$.
\end{remark}

Next, we recall the simplified energy functional which is given for $  \ph \in C_c(X) $ by
\begin{align*}
	\E(\ph)= \sum_{x,y\in X}b(x,y) ({\sol{u}}\otimes {\sol{u}})_{xy} |\nabla_{xy} \varphi|^2 \left[ {|\nabla_{xy} {\sol{u}}|} \langle{|\varphi|}\rangle_{xy}+  ({\sol{u}}\otimes {\sol{u}})_{xy}^{1/2}|\nabla_{xy} \varphi|\right]^{p-2}, 
\end{align*}
where $({\sol{f}}\otimes {\sol{f}})_{xy}:=u(x)u(y)$ and $\langle{|f|}\rangle_{xy}:=(f(x)+f(y))/2$ for any function $f$ and  {where we set $0\cdot \infty =0$ if $1<p<2$.}

In contrast to the ground state transform 
 in the linear case, see e.g. \cite[Proposition~3.18]{KLW}, \cite[Section 4.2]{KPP},  the simplified energy is not a representation of  the energy functional $ \Q $. However, we still have a two-sided estimate that goes back to {Pinchover--Tertikas--Tintarev \cite[Lemma 2.2]{PTT}} in the continuum setting and to Fischer \cite[Theorem 3.1]{Florian_nonlocal} in the discrete setting.   
\begin{proposition}[Simplified energy, {\cite[Theorem 3.1]{Florian_nonlocal}}] \label{Prop:simp_energy} There are $C_1,C_2>0$ such that  for all strictly positive functions  $ \sol{u} \in \mathcal{F}(X)$  and for all  $ \varphi  \in  C_c(X) $
	$$ C_1\E(\varphi)\leq \Q({\sol{ u}}\varphi)-\sum_{X}mu\mathcal{H}[u]|\varphi|^{p}   \leq C_2 \E(\varphi) .$$
\end{proposition}

%

\subsection{Criticality theory and Liouville comparison principle}
In this subsection, we recall some basic notions of quasilinear criticality theory on discrete graphs. We refer to \cite{Florian_cri} for more details on this topic and \cite{KPP} for criticality theory on graphs in the linear setting. 

A function $u\in \mathcal{F}_{b}(X)$ satisfying $\mathcal{H}u\ge 0$ ($\mathcal{H}u\le 0$) is called $\mathcal{H}$-superharmonic ($\mathcal{H}$-subharmonic).  A function $u$ is said to be $\mathcal{H}$-harmonic if $u$ is $\mathcal{H}$-superharmonic and $\mathcal{H}$-subharmonic. 

We say that $\mathcal{H}$ is a positive quasilinear Schr\"odinger operator and write 
\begin{equation*}
	\mathcal{H}\ge 0
\end{equation*}
if $\mathcal{H}$ admits a positive $\mathcal{H}$-superharmonic function. We recall  that $\Q \geq 0$ on $C_c(X)$ if and only if $\mathcal{H}$ admits a positive supersolution on $X$ which is known as the Agmon-Allegretto-Piepenbrink-type theorem \cite[Theorem~2.3]{Florian_cri}. In the case of locally finite graphs, $\Q \geq 0$ on $C_c(X)$ even implies the existence of a positive solution on $X$. By the Harnack inequality \cite[Lemma 4.4]{Florian_nonlocal} it can be seen that due to the connectedness, every positive supersolution is indeed strictly positive.   

\begin{defin}[Critical/subcritical]Let $\Q=\Q_{p,b,m,V}\geq 0$ be the energy functional associated with the quasilinear Schr\"o\-dinger operator ${\H}$. The functional $\Q$ is said to be \emph{subcritical} in $X$ if there is a positive
$W\in C(X)$ such that $Q_{p,b,m,V-W}\ge 0$ on $C_{c}(X)$. Otherwise, the form
	$\Q$ is called \emph{critical}  in $X$.
\end{defin}
{It follows from \cite[Theorem~5.1]{Florian_nonlocal} that  $\Q$ is critical if and only if  there exists a so called \emph{null-sequence} $(\ph_{n})$ of positive functions }in $C_{c}(X)$ that satisfies $\ph_{n}(o)\asymp 1$ for some $o\in X$ and all $n$, and $\Q(\ph_{n})\to 0$ as $n\to\infty$.
It further turns out that $\Q$ is critical in $X$ if and only if the equation $\H[\varphi]=0$ admits a unique positive $\H$-superharmonic function (up to a multiplicative positive constant) 
\cite[Theorem~5.1]{Florian_nonlocal}. In fact, such an $\H$-superharmonic function is a strictly positive $\H$-harmonic and is called an {\em Agmon ground state}.

\begin{definition}[Positive solution of minimal growth at infinity]
		A positive function $ u $ is called a {\em positive solution of $\H[u]=0$ of minimal growth at infinity}  in $X$ if $ u>0 $ and $\H[u]=0$   in~$X\setminus K_{0}$ for some compact $ K_{0}\Subset X $ and for any $ v $ such that $ v>0 $ and  $\H[v]\geq 0$ in~$X\setminus K$  for some $ K_{0}\Subset K\Subset X $ which satisfies~$u\leq v$ on 
			$K$ one has $u\leq v$ in~$X\setminus K$. Such a function $ u $ that satisfies $ u>0 $ and $ \H[u] =0$ on $ X  $ is called a {\em global minimal positive solution}.
	\end{definition}

Now we recall the quasilinear version of the Liouville comparison principle, which will play the key role in proving our main result. In the continuum, this principle goes back to \cite{P_Liou} in the linear case and \cite{PTT} in the quasilinear case, and the linear discrete analogue can be found in \cite{DKP}. The quasilinear version of the Liouville comparison can be found in \cite[Proposition 4.9]{Florian_nonlocal}. For the sake of completeness, we provide a proof.

\begin{thm}[{Liouville comparison principle, 
cf.~\cite[Proposition 4.9]{Florian_nonlocal}}]\label{thm:comparison} 
Suppose $\mathcal{H}$ and $\tilde{\mathcal{H}}$ are positive quasilinear Schr\"odinger operators associated to connected graphs $b$, $\tilde b$ and potentials $V, \tilde V $ over $ (X,m) $. Let $u\in \mathcal{F}_b(X)$ and $v\in \mathcal{F}_{\tilde b}(X)$ be such that
	\begin{itemize}
		\item [(a)] $\tilde{\mathcal{H}}$ is critical and $v>0$ is its Agmon ground state,
		\item [(b)] $u_{+}\neq 0$  and $\mathcal{H}u_{+}\leq0$,
		\item [(c)] for some $ C>0 $, we have {on $\{u>0\}\times \{u>0\}$}
		\begin{align*}
			b^{2/p}(u_{+}\otimes u_{+})&\leq C {\tilde b}^{2/p}(v\otimes v),\\
			b^{1-2/p}{|\nabla u_{+}|^{p-2}}&\leq C {\tilde  b}^{1-2/p}{|\nabla v|^{p-2}}.
		\end{align*}
	\end{itemize}
	Then, $\mathcal{H}$ is critical and $u>0$ is an Agmon ground state of $\mathcal{H}$.
\end{thm}

\begin{proof} In \cite[Theorem 2.1]{DKP}, we have shown the statement for $p=2$. So, we may assume  $p\neq 2$.
By $(a)$ there exists a null-sequence $(\ph_{k}) $ in $ C_{c}(X)$ with respect to a vertex $x\in X$ for the energy functional $ \tilde{\mathcal{Q}} $ of $\tilde{\mathcal{H}}$. For $ k\in\mathbb{N} $, set
	\begin{align*}
		\psi_{k}=\frac{u_{+}\ph_{k}}{v} \, .
	\end{align*}
     Using the simplified energy functional, Proposition~\ref{Prop:simp_energy},  we estimate the energy functional $\mathcal{Q}$ associated to $\mathcal{H}$ as follows 
	\begin{multline*}
		\mathcal{Q}(\psi_{k})= \mathcal{Q}\left(\frac{u_{+}\ph_{k}}{v}\right) \\
         {\lesssim}
        \frac{1}{2}\sum_{X\times X}\!b (u_{+}\otimes u_{+})\!\left|\nabla (\frac{\ph_{k}}{v})\right|^{2} \!\left[|\nabla u_+| \langle |\frac{\varphi_k}{v}|\rangle + (u_{+}\otimes u_{+})^{\frac{1}{2}}\left|\nabla (\frac{\varphi_k}{v})\right|\right]^{p-2} \!\!+\!\!\sum_{X}\!m(u_{+}{\mathcal{H}}u_{+})\!\left|\frac{\ph_{k}}{v}\right|^{p}  \\
        =\frac{1}{2}\sum_{X\times X}b^{\frac{2}{p}} (u_{+}\otimes u_{+})\left|\nabla (\frac{\ph_{k}}{v})\right|^{2} \left[ \left(b^{\frac{p-2}{p}}|\nabla u_+|^{p-2} \right)^{{\frac{1}{p-2}}} \langle |\frac{\varphi_k}{v}|\rangle + \left(b^{\frac{2}{p}}(u_{+}\otimes u_{+}) \right)^{\frac{1}{2}}\left|\nabla (\frac{\varphi_k}{v})\right|\right]^{p-2} \\ 
 +\sum_{X}m(u_{+}{\mathcal{H}}u_{+})\left|\frac{\ph_{k}}{v}\right|^{p}  .
 	\end{multline*}
 Now, observe that for $0\leq \alpha\le \tilde \alpha$, $0\leq \beta\le \tilde \beta$ and $\gamma>-1$, we have
 \begin{equation*}
 	{\alpha^2}\left(\beta^{1/\gamma}+\alpha\right)^{\gamma}\leq 	{\tilde\alpha^2}\left(\tilde\beta^{1/\gamma}+\tilde\alpha\right)^{\gamma}
 \end{equation*}
 which can be checked by considering the left hand side as either a function of $\alpha $ or $\beta$ and observing that these functions have a positive derivative. This observation  now allows us to use assumptions $(b)$, and $(c)$, to estimate
	\begin{align*} 
 	\mathcal{Q}(\psi_{k})
    &\lesssim \frac{1}{2}\!\sum_{X\times X}\! \tilde{b}^{\frac{2}{p}} (v\otimes v) \!\left|\nabla (\frac{\ph_{k}}{v})\right|^{2}\! \left[\left( \tilde{b}^{\frac{p-2}{p}}|\nabla v|^{p-2} \right)^{\frac{1}{p-2}}\!\langle |\frac{\varphi_k}{v}|\rangle \!+\! \left(\tilde{b}^{2/p}(v\otimes v) \right)^{\frac{1}{2}} \left|\nabla (\frac{\varphi_k}{v})\right|\right]^{p-2}\\
  &= \frac{1}{2}\sum_{X\times X} \tilde{b}^{\frac{2}{p}+\frac{p-2}{p}} (v\otimes v) \left|\nabla (\frac{\ph_{k}}{v})\right|^{2} \left[|\nabla v| \langle |\frac{\varphi_k}{v}|\rangle + \left(v\otimes v \right)^{\frac{1}{2}} \left|\nabla (\frac{\varphi_k}{v})\right|\right]^{p-2}\\
		&\lesssim	\tilde{\mathcal{Q}}(\ph_{k})\to 0,\qquad k\to\infty.
	\end{align*}
	Thus, $( \psi_{k} )$ is a null-sequence for $ \mathcal{H} $, and $ \mathcal{H} $ is critical. In particular, 
	$( \psi_{k} )$ converges pointwise to an Agmon ground state for $ \mathcal{H} $. Since $ \tilde{\mathcal{H}} $ is critical and $ v>0 $ is $ \mathcal{H} $-harmonic, the null-sequence $ (\ph_{k}) $ converges pointwise to a positive multiple of $ v>0$. Hence, $ (\psi_k) $ converges pointwise to a positive multiple of $ u_{+} $ which is therefore strictly positive.  Thus, $u_+=u$ is an Agmon ground state for $ \mathcal{H} $. 
\eat{
	By (a) there exists a null-sequence $(\ph_{k}) $ in $ C_{c}(X)$ with respect to a vertex $x\in X$ for the energy function $ \mathcal{Q}' $ of $\mathcal{H}'$. For $ k\in\mathbb{N} $, set
	\begin{align*}
		\psi_{k}=\frac{u_{+}\ph_{k}}{v} \, .
	\end{align*}
	Observe that the ground state transform of the energy functional $\mathcal{Q}$ of $\mathcal{H}$ with respect to $u_+$, (b) and (c) yields 
	\begin{align*}
		\mathcal{Q}(\psi_{k})&= \frac{1}{2}\sum_{X\times X}b (u_{+}\otimes u_{+})|\nabla (\ph_{k}/v)|^{2}+\sum_{X}m(u_{+}\mathcal{H}u_{+})({\ph_{k}}/{v})^{2}\\&\le \frac{C}{2}\sum_{X\times X}b' (v\otimes v)|\nabla (\ph_{k}/v)|^{2}\\
		&= C	\mathcal{Q}'(\ph_{k})\to 0,\qquad k\to\infty.
	\end{align*}
	Thus, $( \psi_{k} )$ is a null-sequence for $ \mathcal{H} $, and $ \mathcal{H} $ is critical. In particular, 
	$( \psi_{k} )$ converges pointwise to an Agmon ground state for $ \mathcal{H} $. Since $ \mathcal{H}' $ is critical and $ v>0 $ is $ \mathcal{H} $-harmonic, the null-sequence $ (\ph_{k}) $ converges pointwise to a positive multiple of $ v>0$. Hence, $ (\psi_k) $ converges pointwise to a positive multiple of $ u_{+} $ which is therefore strictly positive.  Thus, $u_+=u$ is an Agmon ground state for $ \mathcal{H} $. }
\end{proof}
\begin{remark}
	{Let $p\leq 2$. Suppose that 
	$\mathcal{H}$ is a positive quasilinear Schr\"odinger operator associated to  a connected graph $b$ over $ (X,m) $ and let $\Delta_p$ be the corresponding $p$-Laplacian. Assume that  $\Delta_p$ is critical, so, the constant function  $v=1$ is its Agmon ground state. Taking into account $p\le 2$, the comparison in item (c) in the theorem above then reduces to $u$ being bounded. Thus, any bounded  $u\in \mathcal{F}_b(X)$ with $u_{+}\neq 0$ and $\mathcal{H}u_{+}\leq 0$ is an Agmon ground state of $\mathcal{H}$ and $u>0$. In particular,  $\mathcal{H}$ is critical.}
\end{remark}

\section{General Landis Type Theorem}

{A \emph{minimal positive Green function}} $G_{\alpha}$ for $\alpha \geq 0$ and $o\in X$ is the smallest positive function $\phi$ such that  $(\Delta_{p} +\alpha)\phi=1_{o}$, i.e., we have
\begin{equation*}
	(\Delta_{p} +\alpha)G_{\alpha}=1_{o} 
\end{equation*}
and $G_{\alpha}$ is a positive solution of minimal growth at infinity for $ (\Delta_p+\alpha) $. It follows from {\cite[Theorem~2.5]{Florian_cri}} that such a $G_{\alpha}$ exists for $\alpha >0$. For $\alpha =0$, the Green function $G_{0}$ exists if and only if the $\Delta_p$ is subcritical on $X$, {\cite[Theorem~2.5]{Florian_cri}.} 

The next theorem is the abstract main result of the paper, which extends our earlier result \cite[Theorem~3.1]{DKP} concerning the linear case ($p=2$) to $p \in (1,\infty)$.

\begin{thm}\label{t:Landis_general}
	{Let $ \mathcal{H}=\Gd_p+V$ be a positive quasilinear Schr\"odinger operator associated to a connected graph $ b $ over $ (X,m) $ with potential $ V\le 1 $. Let $\tilde {\mathcal{H}}$ be a critical quasilinear Schr\"odinger operator with Agmon ground state $ v>0 $ associated to a connected graph $ \tilde b$ over $ (X,m)$.  
		Any $ \mathcal{H} $-harmonic function $ u $ with non-trivial positive part $u_{+}$ 	satisfying for some  $C>0$	
		 {on $\{u>0\}\times \{u>0\}$}
		\begin{align*}
			b^{2/p}(u_{+}\otimes u_{+})&\leq C {\tilde b}^{2/p}(v\otimes v),\\
			b^{1-2/p}{|\nabla u_{+}|^{p-2}}&\leq C {\tilde  b}^{1-2/p}{|\nabla v|^{p-2}},
		\end{align*}
		and
			\begin{equation*}
				 \liminf_{x\to\infty} \frac{u_{+}(x)}{G_{1}(x)}=0 \,,
			\end{equation*}
	is trivially zero.}
\end{thm}

The following lemma generalizes \cite[Lemma 3.2]{DKP} to the quasilinear setting, assuming additionally that one of the subharmonic functions is the trivial function $v=0$.
\begin{lemma}\label{lem:u_+}
Let $ \mathcal{H}=\Gd_p+V$ be a positive quasilinear Schr\"odinger operator associated to a  graph $ b $ over $ (X,m) $ with potential $V$.
If $u$ is an $ \mathcal{H} $-subharmonic function, then $u_+:=\max\{u, 0\}$ is also $ \mathcal{H} $-subharmonic.
\end{lemma}
\begin{proof} Let $u$ be $ \mathcal{H} $-subharmonic. Fix $x \in X$.

Case~1: $u_+(x) =0$. Then,
\begin{align*}
    (\Delta_p + V)[u_+](x) = -\frac{1}{m(x)}\sum_{y \in X} b(x,y) | u_+(y)|^{p-2} u_+(y) \leq 0
\end{align*}

Case~2: $u_+(x) >0$. That is $u_+(x)=u(x) > 0$. Then,
\begin{align*}
    (\Delta_p + V)[u_+](x) 
& = \frac{1}{m(x)}\sum_{y \in X} b(x,y) (u(x)-u_+(y))^{\langle p-1\rangle} +  V(x) {|u(x)|^{p-2}u(x)} \,.
\end{align*}
Now, two situations may occur. First, $u_+(y)=0$, that is, $u(y) \leq 0$. In this case, since $u(x)>0$ it follows that  
$$(u(x)-u_+(y))^{\langle p-1\rangle}=u(x)^{p-1}  \leq (u(x)-u(y))^{p-1}=(u(x)-u(y))^{\langle p-1\rangle}.$$ On the other hand, if $u_+(y) >0$ then $u_+(y)=u(y)>0$, hence  
$$(u(x)-u_+(y))^{\langle p-1\rangle}= (u(x)-u(y))^{\langle p-1\rangle}.$$

Thus, it follows that 
\begin{align*}
 (\Delta_p + V)[u_+](x)
& \leq 
 (\Delta_p + V)[u](x) \leq 0 \,.
\end{align*}
Consequently, the above conclusion holds in this case as well.
This shows that $u_+$ is $ \mathcal{H} $-subharmonic.
    \end{proof}
\begin{remark}\label{ccor_subsol}
In the linear case $p=2$, the conclusion of Lemma~\ref{lem:u_+} can be extended to say that the pointwise absolute value of an $\mathcal{H}$-harmonic function is $\mathcal{H}$-subharmonic. This is still  true in the quasilinear case $p\neq 2$ and we discuss it here shortly although we do not apply it in this paper:
If $u$ is an $ \mathcal{H} $-harmonic function, then $|u|$ is $ \mathcal{H} $-subharmonic.
\begin{proof} Define the function $f(s) = s |s|^{p-2}$, $s\in \mathbb{R}$, which is strictly increasing for $p > 1$. Observe that  for $x\in X$
	\begin{equation*}
	\Delta_{p}[|u|](x)=	\sum_{y\in X}b(x,y)f(\nabla_{xy}|u|)
	\end{equation*}
	\\
	Case 1: $u(x)\ge 0$. Then $|u(x)|=u_+(x)$, and the monotonicity of $f$ implies
	$$f(\nabla_{xy}|u|)= f(u_+(x)-|u(y)|)\leq f(u_+(x)-u_+(y))=f(\nabla_{xy}u_+).$$ 
	Hence, using this and Lemma~\ref{lem:u_+}, we obtain 
	\begin{equation*}
	\Delta_{p}[|u|] =\sum_{y\in X}b(x,y) f(\nabla_{xy}|u|)\leq \sum_{y\in X}b(x,y) f(\nabla_{xy}u_+)\leq -V[u_+](x)=-V[|u|](x).
	\end{equation*}
	 Case 2: $u(x)< 0$. Then  $|u(x)|=(-u)_+(x)$, and the monotonicity of $f$ implies
	 $$f(\nabla_{xy}|u|)= f((-u)_+(x)-|u(y)|)\leq f((-u)_+(x)-(-u)_+(y))=f(\nabla_{xy}(-u)_+).$$ 
	 Since $u$ is $\mathcal{H}$-harmonic, so is $-u$, the latter inequality and Lemma~\ref{lem:u_+} with $-u$ imply 
	 \begin{equation*}
	 \Delta_{p}[|u|] \!=\!\sum_{y\in X}b(x,y) f(\nabla_{xy}|u|)\!\leq \!\sum_{y\in X}b(x,y) f(\nabla_{xy}(-u)_+)\!\leq\! -V[(-u)_+](x)=-V[|u|](x).\qedhere
	 \end{equation*}
\end{proof}\end{remark}
\begin{proof}[{Proof of Theorem~\ref{t:Landis_general}}]
Let $ u $ be $ \mathcal{H} $-harmonic. {By our assumption, $ u_{+}\neq 0 $.} Then, by  Lemma~\ref{lem:u_+} above, we have
\begin{align*}
	\mathcal{H}u_{+}\le 0.
\end{align*}
We collect the following facts:
\begin{itemize}
	\item [(a)] $\tilde{\mathcal{H}}$ is critical and $v>0$ is an Agmon ground state.
	\item [(b)] $u_{+}\neq  0$  and $\mathcal{H}u_{+}\leq 0$.
	\item [(c)] {$ b^{2/p}(u_{+}\otimes u_{+})\leq C \tilde{b}^{2/p}(v\otimes v)$ for some $ C>0 $.} 
    \item [(d)] 			$b^{1-2/p}{|\nabla u_{+}|^{p-2}}\leq C {\tilde  b}^{1-2/p}{|\nabla v|^{p-2}}$ for some $ C>0 $  {on $\{u>0\}\times \{u>0\}$}.
\end{itemize}
Thus, by the Liouville comparison theorem, Theorem~\ref{thm:comparison}, we infer that $ \mathcal{H} $ is critical and $ u=u_{+}>0 $ is an Agmon ground state for $ \mathcal{H} $. In particular, $\mathcal{H} u=0$.
Now, since $ V
\leq 1 $, we have 
\begin{align*}
(\Delta_p+1)[ u] \ge (\Delta_p+V) [u]=\mathcal{H} [u]=0.
\end{align*}
As $ G_{1} $ is a solution of minimal growth at infinity for $ \Delta_p+1 $, we obtain that $G_1\leq Cu$ in $X$ for some $C>0$ which contradicts our assumption  $\liminf_{x\to\infty} {u(x)}/{G_{1}(x)}=0 $.  Thus, any such $ \mathcal{H} $-harmonic function is trivially zero.    
\end{proof}

\eat{ \begin{proof}[{Proof of Theorem~\ref{t:Landis_general}}]	Let $ u $ be a $ \mathcal{H} $-harmonic function, and assume without loss of generality that $ u_{+}\neq 0 $ (otherwise, consider $ -u$). Then, by  Lemma~\ref{lem:v_+} above, we have
	\begin{align*}
		\mathcal{H}u_{+}\le 0.
	\end{align*}
	We collect the following facts:
	\begin{itemize}
		\item [(a)] $\mathcal{H}'$ is critical and $v>0$ is an Agmon ground state.
		\item [(b)] $u_{+}\neq  0$  and $\mathcal{H}u_{+}\leq 0$.
		\item [(c)] {$ b(u_{+}\otimes u_{+})\leq C b'(v\otimes v)$ for some $ C>0 $.} 
	\end{itemize}
	Thus, by the Liouville comparison theorem, Theorem~\ref{thm:comparison}, we infer that $ \mathcal{H} $ is critical and $ u=u_{+}>0 $ is an Agmon ground state for $ \mathcal{H} $. In particular, $\mathcal{H} u=0$.
	
	Now, since $ V
	\leq 1 $ we have 
	\begin{align*}
		(\Delta+1) u \ge (\Delta+V) u=\mathcal{H} u=0.
	\end{align*}
	As $ G_{1} $ is a solution of minimal growth at infinity for $ \Delta+1 $, we obtain that $G_1\leq Cu$ in $X$ for some $C>0$ which contradicts our assumption  $\liminf_{x\to\infty} {u(x)}/{G_{1}(x)}=0 $.  Thus, any such $ \mathcal{H} $-harmonic function is trivially zero.
    \end{proof}}

Often it is hard to get a hold of asymptotics $G_{1}$ in the quasilinear case. Nevertheless, it is easier to find the asymptotics of $G_{0}$ if it exists. Having this in mind, in the following corollary, we replace $G_1$ from the liminf condition of Theorem \ref{t:Landis_general} by a positive solution $g$ of $\Delta_p=0$ in $X$ of minimal growth at infinity. 
\begin{cor} \label{cor:neg_pot}
		{Let $ \mathcal{H}=\Gd_p+V$ be a positive quasilinear Schr\"odinger operator associated to a connected graph $ b $ over $ (X,m) $ with potential $ V\le 0 $. Let $\tilde {\mathcal{H}}$ be a critical quasilinear Schr\"odinger operator with Agmon ground state $ v>0 $ associated to a connected graph $ \tilde b$ over $ (X,m)$.  
		Any $ \mathcal{H} $-subharmonic function $ u $ with non-trivial positive part $u_{+}$ 	satisfying for some  $C>0$	
		 {on $\{u>0\}\times \{u>0\}$}
		\begin{align*}
			b^{2/p}(u_{+}\otimes u_{+})&\leq C {\tilde b}^{2/p}(v\otimes v),\\
			b^{1-2/p}{|\nabla u_{+}|^{p-2}}&\leq C {\tilde  b}^{1-2/p}{|\nabla v|^{p-2}},
		\end{align*}
		and for some positive   {$p$-harmonic} function $g$ of minimal growth at infinity 
		\begin{equation*}
			\liminf_{x\to\infty} \frac{u_{+}(x)}{g(x)}=0
		\end{equation*}
		is trivially zero.}
\end{cor}
\begin{proof}
    Using the same arguments as in the proof of Theorem \ref{t:Landis_general}, we can prove that $u>0$ and $u$ is an Agmon ground state.  Now, since $ V
\leq 0 $ we have 
\begin{align*}
\Delta_p u \ge (\Delta_p+V) u=\mathcal{H} u=0.
\end{align*}
As $ g $ is a positive solution of minimal growth at infinity for $ \Delta_p $, we obtain that $g \leq Cu$ in $X$ for some $C>0$ which contradicts our assumption  $\liminf_{x\to\infty} {u_{+}(x)}/{g(x)}=0 $.  Thus, any such $ \mathcal{H} $-harmonic function is trivially zero.
\end{proof}

\section{Examples}\label{sec:Examples}

\subsection{Model Graphs}
Let $b$ be a graph on $X$. The combinatorial graph distance $d: X \times X \rightarrow [0,\infty)$ is $d(x,y)$ is the minimal number of edges of a path connecting $x$ and $y$. For a fixed vertex $o \in X$, we denote for $x\in X$
\begin{equation*}
	|x|=d(x,o)
\end{equation*}
and
define the ball centered at $o$ of radius $r>0$ by
$$B_r(o)=\{x \in X\mid |x| \leq r\}$$
as well as  the corresponding sphere by 
$$S_r(o)=\{x \in X\mid |x| = r\} \,.$$
Observe that any neighbor $y$ of a vertex $x   \in S_r(o)$ is contained in either $S_{r-1}(o)$, $S_r(o)$ or $S_{r+1}(o)$. The inner and outer curvature are respectively defined for $x\in X$ as
$$k_{\pm}(x)=\frac{1}{m(x)} \sum_{y \in S_{|x|  \pm 1}(o)} b(x,y).
$$
We define $k_{-}(o)=0$. Now, $b$ is said to be a model graph on $X$ with respect to $o$ if $k_{\pm}$  are spherically symmetric, i.e., $k_{\pm}(x)=k_{\pm}(y)$ 
for all $x,y\in X$ with $|x|=|y|$. Furthermore, we let
 for every $k \geq 0$, 
$$\partial_bB_k(o) :=
\sum_{x\in S_k(o), y\in S_{k+1}(o)} b(x,y),$$
i.e., $\partial_bB_k(o)$ denotes the total edge weight between spheres of radius $k$ and spheres of
radius $k+1$ with respect to the root $o\in x$.  The following proposition follows by a direct computation and can also be found in  \cite{F_Thesis}.

\begin{proposition}[Example 9.8 in \cite{F_Thesis}]\label{prop_sub_crt}
	The energy functional of the $p$-Laplacian on a model graph $b$ is subcritical if and only if for some (all) $x\in X$
\begin{equation*}
		G_0(x)=\sum_{k=|x|}^{\infty} \left( \frac{m(o)}{\partial_b B_k(o)}\right)^{1/(p-1)}
\end{equation*}
	is finite. In this case, $G_{0}$ is the minimal positive Green function.
\end{proposition}
In the case where $G_{0}$ is finite we say that the model graph $b$ is subcritical. We denote $f\in O(g) $ if $|f|\lesssim g$ for functions $f$ and $g$.
\begin{thm}\label{t:model}
Let $u$ be an $ \mathcal{H} $-harmonic function  of a positive $p$-Schr\"odinger operator $ \mathcal{H} =\Delta_p+V  $ with $ V\leq 0 $ on a subcritical model graph $b$ over $(X,m)$ such that  
\begin{align*}
   u\in  O\left(G_0 ^{(p-1)/p} \right)\,, 
\end{align*}
\begin{align*}%
  |\nabla_{xy} u_+|^{p-2} \in O \left( { {\left(G_{0}(z)^{-\frac{1}{p}} (\partial_b B_{|x|}(o))^{-\frac{1}{p-1}} \right)^{p-2}}}\right)
\end{align*}
  for  
   $x,y\in \{u>0\}$,
 $ x\sim y$ { with }$|x|>|y|$, where $z=x$ for $p<2$ and
  {$z=y$ for $p\ge 2$},  as well  as $  |\nabla_{xy} u_+|=0$ for $x\sim y$ with $|x|=|y|$ for $p\ge 2$ and
\begin{align*}%
 \liminf_{x\to\infty}\frac{|u(x)|}{G_0(x)} =0.
\end{align*}
Then,  $ u=0 $.
\end{thm}
\begin{proof}
{Let $\Phi=G_0^{(p-1)/p}$.} In  \cite[Example 12.18]{F_Thesis}   an {\it optimal Hardy-weight} $W_{\mathrm{op}}$ is constructed for $\Delta_p$ via 
{$W_{\mathrm{op}}= \Delta_{p}\Phi/\Phi^{(p-1)}$.} 
This ensures that  $\Delta_p-W_{\mathrm{op}}$ is critical with the Agmon ground state 
{$\Phi_{\Delta_p-W_{\mathrm{op}}}=\Phi$. Notice that $u_+ \in O(\Phi)$.}

Let $x\sim y$. If $|x|=|y|$, then $\nabla_{xy} \Phi=0$. Assume now $|x|>|y|$ which then implies $G_{0}(y)>G_{0}(x)$. Hence, by the mean value theorem applied to the function $t\mapsto t^{1-1/p}$, we obtain
\begin{equation*}
	|\nabla_{xy}\Phi|=\nabla_{{yx}}G_{0}^{1-\frac{1}{p}}\le\frac{(p-1)}{p}G_{0}(x)^{-\frac{1}{p}}\nabla_{{yx}}G_{0}=\frac{(p-1)}{p}m(o)^{\frac{1}{(p-1)}} G_{0}(x)^{-\frac{1}{p}}(\partial_b B_{|x|}(o))^{-\frac{1}{p-1}}.
\end{equation*}
Similarly,
\begin{equation*}
	|\nabla_{xy}\Phi|=\nabla_{{yx}}G_{0}^{1-\frac{1}{p}}\ge\frac{(p-1)}{p}G_{0}(y)^{-\frac{1}{p}}\nabla_{{yx}}G_{0}=\frac{(p-1)}{p}m(o)^{\frac{1}{(p-1)}} G_{0}(y)^{-\frac{1}{p}}(\partial_b B_{|x|}(o))^{-\frac{1}{p-1}}.
\end{equation*}
Thus, $|\nabla_{xy} u_+|^{p-2} \in O(|\nabla_{xy} \Phi|^{p-2})$ for $x \sim y$ where we first use the {lower} bound for $p\ge 2$ and secondly the {upper} bound for $p<2$.
Also, we note that 
$\liminf_{x \rightarrow \infty} u_+(x)/G_0 (x)=0$. Thus, $u=0$ by Corollary~\ref{cor:neg_pot}.
\end{proof}

\begin{example}
	An anti-tree is a model graph with standard weights such that $b\vert_{S_{r}(o)\times S_{r+1}(o)}=1$ for all $r\ge 0$ and $b=0$ otherwise. We also set $m=1$.
	Let $p>1$, $\gamma>0$ and an anti-tree be given such that 
	\begin{equation*}
		\# S_{r+1}(o)=[r^{\gamma}],
	\end{equation*}  
	where  $[s]$ denotes the smallest integer larger or equal to $s\in \mathbb{R}$. 
	In this case, we have  for all $r\ge0$
	 $$\partial_{b}B_{r}(o)=
	\#S_{r}(o)\cdot \#S_{r+1}(o) 	\asymp 
	r^{2\gamma}.$$ Hence, we obtain  for  $|z|\ge 1$
	$$G_{0}(z)=\sum_{k=|z|}^{\infty} \left( \frac{1}{\partial_b B_k(o)}\right)^{1/(p-1)}\asymp \sum_{k=|z|}^{\infty} k^{-\frac{2\gamma}{p-1}}  \asymp |z|^{-\frac{2\gamma}{(p-1)}+1}=|z|^{-\frac{2\gamma -p+1}{(p-1)}}$$
and 	 the anti-tree is subcritical if $\gamma>(p-1)/2$. In this case,  we have for $z\in\{x,y\}$
    \begin{equation*}
		G_{0}(z)^{-\frac{1}{p}}    {(\partial_b B_{|x|}(o))^{-\frac{1}{(p-1)}}}\asymp |z|^{\frac{2\gamma -p+1}{p(p-1)}}{|x|^{-\frac{2\gamma}{(p-1)}}} \asymp |x|^{-\frac{(2\gamma -1)}{p}}.
	\end{equation*}
Hence, using the above theorem, we obtain that for any $ \mathcal{H} $-harmonic function $u$ of a positive $p$-Schr\"odinger operator $ \mathcal{H} =\Delta_p+V  $ with $ V\leq 0 $ on a subcritical anti-tree with parameter $\gamma>(p-1)/2$ such that 
\begin{align*}
	u\in  O\left(|x| ^{-\frac{2\gamma-p+1}{p}} \right)\,, 
\end{align*}\begin{align*}%
	|\nabla_{xy} u_+|^{p-2} \in O \left(  |x|^{-\frac{(2\gamma -1)}{p} }\right)
\end{align*}
for  
   $x,y\in \{u>0\}$,
 $ x\sim y$ { with }$|x|>|y|$ and

\begin{align*}%
	\liminf_{x\to\infty}{|u(x)|}||x|^{\frac{2\gamma -p+1}{(p-1)}} =0
\end{align*}
is trivial.
\end{example}

\subsection{Regular Trees} In this section we look at a special class of model graphs that allow for even more explicit criteria. In particular, for these trees we can treat the case $V\leq 1$ and not only $V\leq 0$ as above.
We consider a tree  with countably infinite vertex set $ X=\mathbb{T}_d $. We fix an arbitrary vertex $ o $ and denote by $ |x|= d(x,o) $ the {\it combinatorial graph distance} to $ o $ from a vertex $ x $. A tree is said to be {\it {$d$}-regular} if $k_+(x)=d$ and $k_-(x)=1$ for every $x\in \mathbb{T}_d$.
Again, we consider the Laplacian $ \Delta $ with standard weights $ b(x,y)=1 $ if $d(x,y)= 1$ and $0$ otherwise, as well as  $ m=1 $.

Recall that a function on $X$ is called  \emph{spherically symmetric} if  $u(x)=u(y)$ whenever $|x|=|y|$. In this case we write with slight abuse of notation $u(x)=u(|x|)$ for $x\in X$.

\begin{lem}
	Let $p>1$ and let a $d$-regular tree be given. Then there is $\beta=\beta_{p,d} \in (0,1)$ which satisfies
	\begin{align*}
		(1-\beta)^{p-1} = \left[ \frac{1}{d} \left(\frac{1}{\beta} -1\right)^{p-1} - \frac{1}{d} \right] \,.
	\end{align*}
	such that $ u(|x|)=\beta^{|x|}$ satisfies  $(\Delta_p +1)[u]=0$ in 
    $X \setminus \{o\}$ and has minimal growth at infinity.
\end{lem}
\begin{proof}
	Let $\beta\in (0,1)$ be arbitrary and $u(x)=\beta^{|x|}$ for $x\in X$. Then, for $X\setminus \{o\}$ we have
	\begin{align*}
		(\Delta_p +1)[u](|x|) & =d\left(u(|x|)-u(|x|+1) \right)^{\langle p-1\rangle} + \left(u(|x|)-u(|x|-1) \right)^{\langle p-1\rangle} + u(|x|)^{p-1} \\
	& = du(|x|)^{p-1} \left(1-\beta \right)^{\langle p-1\rangle}  + u(|x|)^{p-1} \left(1-\frac{1}{\beta} \right)^{\langle p-1\rangle} + u(|x|)^{p-1} \\
		& = u(|x|)^{p-1} \left[ 1+ d\left(1-\beta \right)^{\langle p-1\rangle}  +  \left(1-\frac{1}{\beta} \right)^{\langle p-1\rangle}  \right] \,.
	\end{align*}
	Thus, $ (\Delta_p +1)[u](|x|)=0$ if and only if
$$f(\beta):=(1-\beta)^{p-1} - \left[ \frac{1}{d} \left(\frac{1}{\beta} -1\right)^{p-1} - \frac{1}{d} \right] =0\,.$$
One can verify that such $\beta=\beta_{p,d}$ exits because $f(1^-)>0$ and $f(0^+)<0$.\\
Now we will show that  $u$  has minimal growth at infinity. For this we will use what is typically called a Khasminskii-type criterion.  Let $v$ be a positive superharmonic of $\Delta_{p}+1$ in $K$, where $K$ is a compact set containing $o$ such that $u\leq v$ in $K$. Fix an exhaustion $(K_n)$ of finite subsets of $X$. Since  $u(|x|) \rightarrow 0$ as $|x| \rightarrow \infty$, it follows that for any $\vge>0$ there exists $n(\vge)\in \N$  such that
\begin{align*}
    u & \leq  v + \varepsilon \ \ \ \ \mbox{on} \ \  K,   \\
     u & \leq  v + \varepsilon \ \ \ \ \mbox{on} \ \ X \setminus K_n \quad  {\mbox{ for } n \geq n(\vge).}
\end{align*}
 It can be easily verified that for the potential chosen as the  constant function $1$ acting on functions $f$ as $1[f]=f|f|^{p-2}$
 \begin{align*}
   &\, (\Delta_p+1) [v + \varepsilon] = \Delta_p [v] + 1[v + \varepsilon] \geq -1[v]  +1[v + \varepsilon]
    \geq 0 
\end{align*}
on $X$.
This implies $v + \varepsilon$ is a positive supersolution of $(\Delta_p+1)[v]=0$ in $K_n \setminus K$. Then by the weak comparison principle \cite[Lemma 7.1]{F_Thesis} we obtain $u \leq v + \varepsilon$ in $K_n \setminus K$ for all $n \geq n(\vge)$. Hence, $u \leq v+\vge$ in $X \setminus K$, and since $\vge$ is arbitrarily small, it follows that $u \leq v$ in $X \setminus K$. Hence $u$ is a solution of minimal growth in 
$X\setminus\{o\}$.
\end{proof}
\begin{thm}\label{t:trees} 
Let  $d\geq 2$ and let $u$ be an $ \mathcal{H} $-harmonic function  of a positive $p$-Schr\"odinger operator $ \mathcal{H} =\Delta_p+V  $ with $ V\leq 1 $ on a $d$-regular tree $\mathbb{T}_d$. If  
\begin{align*}
\hspace{-1cm} |u| & \in 
O \left( {d^{\,- \frac{|x|}{p}}}\right) \\
    |\nabla_{xy} u_+|  & 
    \begin{cases}
      \ \lesssim \ {d^{\,- \frac{|x|}{p}}} \ \ \mbox{if} \ \ p\geq 2, \\[2mm] 
     \ \gtrsim \ {d^{\,- \frac{|x|}{p}}} \ \ \mbox{if} \ \ p< 2, 
\end{cases} \ \ \ \mbox{for} \ \ x\sim y\,,
\end{align*}
and in addition, for $\beta=\beta_{p,d}$ taken from the lemma above 
$$  \quad \liminf_{x\to\infty}\frac{|u(x)|}{\beta^{|x|}}  =0,$$
then  $ u=0 $.
\end{thm}
\begin{proof}
In  \cite[Example 12.17]{F_Thesis}  an { optimal Hardy-weight} $W_{\mathrm{op}}$ is constructed for $\Delta_p$ on a $d$-regular tree $\mathbb{T}_d$. This ensures $\Delta_p-W_{\mathrm{op}}$ is critical with Agmon ground state 
$\Phi=\Phi_{\Delta_p-W_{\mathrm{op}}}={G_0^{{(p-1)}/{p}}}$, where 
$$G_0(x)=C_{p,d} d^{\,-\frac{|x|}{(p-1)}}$$
for some $C_{p,d}>0$ and all $x\in \mathbb{T}_{d}$. One observes that $u_+ \in O(\Phi)$. 
To verify that $|\nabla_{xy} u_+| \lesssim |\nabla_{xy} \Phi|$ (resp. $|\nabla_{xy} u_+| \gtrsim |\nabla_{xy} \Phi|$) if $p \geq 2$ (resp. $p<2$) for $x\sim y$, it is enough to observe that $y \sim x$ implies $|y|=|x|\pm 1$ and 
$$|\nabla_{xy} \Phi| = C_{p,d}^{1/p'} |{d^{\,-\frac{|x|}{p}}}-{d^{\,-\frac{|x|\pm 1}{p}}}| \asymp {d^{\,-\frac{|x|}{p}}} \,.$$ Also, we note that 
$\liminf_{x \rightarrow \infty} u_+(x)/G_1(x) =0$. Consequently, it follows from Theorem \ref{t:Landis_general} that $u=0$.
\end{proof}

\subsection{$p$-recurrent graphs with  $p\leq 2$} 

In this section we shortly discuss how we can employ our results in the case of recurrent graphs. We restrict to the case $p\le 2$ and $V\leq 0$ outside a compact set.
Recall that a graph is called recurrent if the constant function $1$ is the only positive superharmonic function of the $p$-Laplacian $\Delta_p$ on the graph. 

\begin{theorem}
Let $p\le2$ and let a recurrent graph be given. Let $u$ be an $ \mathcal{H} $-harmonic function  of a positive $p$-Schr\"odinger operator $ \mathcal{H} =\Delta_p+V  $ with $ V\leq 0 $ outside a compact set. If $u$ is bounded  and 
$$\liminf_{|x|\rightarrow \infty} u(x) =0 \,,$$
then $ u=0 $.
\end{theorem}
\begin{proof}
Assume without loss of generality that  $u_+\neq 0$. Let $v=1$ be the constant function. Then, $v$ is an Agmon ground state of $\Delta_p$ and $\Delta_p v=0$. The assumption $p\le 2$ gives that the gradient estimate in Theorem~\ref{thm:comparison}~(c) is trivially satisfied. Thus, by the Liouville comparison principle, Theorem \ref{thm:comparison}, we infer that $u>0$ and $u$ is an Agmon ground state of $\mathcal{H}$. Now, outside a compact set we have
\begin{align*}
\Delta_p [u] =(\Delta_p +0)[u] \ge (\Delta_p+V) [u]=\mathcal{H} [u]=0.
\end{align*}
Since, $v$ is a solution of minimal growth at infinity for $\Delta_p$, we obtain that $v\leq Cu$ in $X$ for some $C>0$ which contradicts our assumption  $\liminf_{|x|\to\infty} u(x)=0 $.  Thus, any such $ \mathcal{H} $-harmonic function is trivially zero.
\end{proof}



\begin{center}
	{\bf Acknowledgments}
\end{center}
U.D. is partially supported by the Basque Government through the BERC 2022-2025 program and by the Spanish Ministry of Science and Innovation: BCAM Severo Ochoa accreditation CEX2021-001142-S/MICIN/AEI/\allowbreak10.13039/501100011033 and CNS2023-143893.
M.K. is grateful for the generous hospitality at the Technion and the financial support by the Swiss Fellowship during the time this research was initiated. Additionally, {the authors acknowledge} the support of the DFG supporting a stay of U.D. at the University of Potsdam.

\end{document}